\documentclass[a4paper,11pt,reqno]{amsart}
\usepackage{amsthm}
\usepackage{amsmath}
\usepackage{enumerate}
\usepackage{amsfonts}
\usepackage{amssymb}
\usepackage{fullpage}
\usepackage{amsmath,amscd}
\usepackage{stmaryrd}
\usepackage{graphicx}
\usepackage{array}
\usepackage{amsmath,amssymb,amsfonts,dsfont}
\usepackage[utf8]{inputenc} 
\usepackage[english]{babel}
\usepackage[T1]{fontenc} 
\usepackage{graphicx}
\usepackage{float}
\usepackage{pdfpages}
\usepackage[a4paper, margin = 3cm, bottom = 3cm]{geometry}
\usepackage{ifpdf}
\usepackage{marginnote}
\usepackage{hyperref}	
\hypersetup{pdfborder=0 0 0, 
	    colorlinks=true,
	    citecolor=black,
	    linkcolor=blue,
	    urlcolor=red,
	    pdfauthor={Guillaume Tahar}
	   }
\newtheorem{thm}{Theorem}[section]
\newtheorem{cor}[thm]{Corollary}
\newtheorem{prop}[thm]{Proposition}
\newtheorem{lem}[thm]{Lemma}
\theoremstyle{definition}
\newtheorem{defn}[thm]{Definition}
\theoremstyle{remark}

\theoremstyle{definition}

\theoremstyle{definition}
\newtheorem{ex}[thm]{Example}
\theoremstyle{definition}

\numberwithin{equation}{section} 
\title{Geometric decompositions of surfaces with spherical metric and conical singularities}
\author{Guillaume Tahar}
\address[Guillaume Tahar]{Faculty of Mathematics and Computer Science, Weizmann Institute of Science,
Rehovot, 7610001, Israel}
\email{tahar.guillaume@weizmann.ac.il}
\date{April 23, 2021}
\keywords{Geodesic arcs, Conical singularities, Constant positive curvature, Triangulations, Core}
\begin{document}
\begin{abstract}
We prove that any compact surface with constant positive curvature and conical singularities can be decomposed into irreducible components of standard shape, glued along geodesic arcs connecting conical singularities. This is a spherical analog of the geometric triangulations for flat surfaces with conical singularities. The irreducible components include not only spherical triangles but also other interesting spherical polygons. In particular, we present the class of \textit{half-spherical concave polygons} that are spherical polygons without diagonals and that can be arbitrarily complicated. Finally, we introduce the notion of core as a geometric invariant in the settings of spherical surfaces. We use it to prove a reducibily result for spherical surfaces with a total conical angle at least $(10g-10+5n)2\pi$.\newline
\end{abstract}
\maketitle
\setcounter{tocdepth}{1}
\tableofcontents

\section{Introduction}

A natural generalization of uniformization theorem concerns surfaces with prescribed singularities. For any compact Riemann surface $S$ of genus $g$ and any real divisor $\sum_{i=1}^{n} \alpha_{i}P_{i}$ formed by positive real numbers $\alpha_{i}>0$ and points $P_{i} \in S$, we may ask if there is a metric with conical singularities of angle $2\pi\alpha_{i}$ at each point $P_{i}$ and constant curvature elsewhere.\newline
If there is a metric with constant curvature $K$ on $S \setminus \{P_{1},\dots,P_{n}\}$, then angle defect at the conical singularities is interpreted as singular curvature in a generalized Gauss-Bonnet formula:\newline
$$\frac{K.Area(S)}{2\pi}=2-2g-n+\sum_{i=1}^{n} \alpha_{i}$$
The sign of this latter quantity, that depends only on the genus of the surface and the prescribed angles, determines if the constant curvature is negative, zero or positive. The metric will, depending on the case, belong to the realms of spherical, flat or hyperbolic geometry.\newline

The case $K \leq 0$ has been completely settled in the work of Picard or Heins, McOwen and Troyanov for modern proofs, see \cite{Pi,He,MO,Tr1,Tr3}. The case of spherical metrics (when $K>0$) is still wide open. Nontrivial obstructions have been raised and partial constructions have been done in various works, see \cite{CW,E,EG,LT,MP,Tr2,UY}.\newline

In this paper, instead of considering globally such metrics, we ask whether any surface with spherical metric and conical singularities can be obtained by gluing several irreducible components.\newline

It is known that in every compact surface with a flat metric and conical singularities, there is a triangulation whose vertices are the conical singularities and edges are geodesic segments, see Lemma 2.2 in \cite{Ta} for a proof of this fact. Triangles of the triangulation are isometric to usual triangles of the plane. The proof of this result generalizes easily for metrics with negative curvature. However, the problem is quite different in spherical geometry. Indeed, there can be continuous of families of geodesic arcs connecting a given pair of conical singularities.\newline

Our main result (Theorem 1.3) is a decomposition theorem. In contrast with triangulation theorem in the flat settings, not only spherical triangles are components of this decomposition. Beforehand, we state some definitions.

\begin{defn}
A \textit{spherical surface} is a compact Riemannian surface with constant positive curvature, conical singularities (at least one) and geodesic boundary (possibly empty) with at least one singularity on each boundary component. Marked points on the boundary are just counted as conical singularities of angle $\pi$.
\end{defn}

We normalize the metrics (by a scaling) in such a way curvature $K=1$. Equivalently, spherical surfaces are (outside conical singularities) locally isometric to the sphere of radius $1$ (whose great circles are of length $2\pi$).\newline

\begin{defn}
For a Riemannian metric $s$ of constant curvature $K$ on a surface $S$, around any interior conical singularity of angle $2\pi\alpha$ at $x \in S$, there is a local coordinate $z$ for which $ds=\frac{2\alpha|z|^{\alpha-1}|dz|}{1+K|z|^{2}}$.\newline
Similarly, around a conical singularity of angle $2\pi\beta$ on the geodesic boundary, the local model is that of an angular sector of angle $2\pi\beta$ bounded by two geodesic arcs.
\end{defn}

There is also a geometric definition of (normalized) spherical surfaces. They are characterized by an atlas of charts on the surface punctured at the singularities with values in the sphere of radius $1$ and such that transition maps belong to $SO(3)$. Particular attention will be given to monodromy of this geometric structure, see Subsection 2.2.\newline
This geometric structure provides also an equivalent definition of geodesic curves, a curve in a spherical surface is geodesic if it is locally the pullback of great circles in spherical charts. We will work only with geodesic arcs without self-intersection.\newline

\begin{thm}
For any spherical surface, there is a system of disjoint geodesic arcs between conical singularities whose connected components of the complement are:\newline
(i) Spherical triangles with angles strictly smaller than $\pi$;\newline
(ii) Foliated digons of arbitrary angle possibly with marked points on one boundary arc (see Subsection 2.5);\newline
(iii) Half-spheres bounded by a great circle with marked points (with no pair of antipodal marked points in the boundary);\newline
(iv) Half-spherical concave polygons (see Definition 3.11).
\end{thm}

The proof of Theorem 1.3 is in the end of Section 3. It follows from a more specific theorem about characterization of spherical polygons without diagonals (Theorem 3.13).\newline

In Proposition 4.18 of \cite{MP}, Mondello and Panov prove existence of the so-called Morse-Delaunay decomposition of spherical surfaces (without boundary). However, this decomposition is obtained by optimization instead of an iterative cutting process.\newline

The decomposition theorem is the starting point of a generalization of the notion of core of flat surfaces (see \cite{HKK,Ta1}) in the settings of spherical surfaces. In Section 4, we prove some technical results in order to make the core an efficient geometric invariant of spherical surfaces. In particular, the core is a locus that can be geometrically decomposed into triangles (see Corollary 4.9).\newline
We use the notion of core to prove a reducibility result for spherical surfaces whose total conical angle is large enough (see Theorem 4.12). 

\section{Basic notions}

We state a Gauss-Bonnet formula for spherical surfaces with boundary (normalized with curvature equal to $1$). For a spherical surface $S$ of genus $g$ with $n$ interior conical singularities of angles $2\pi\alpha_{i}$, $p$ boundary components with $m$ conical singularities of angles $2\pi\beta_{j}$ in the boundary, we have:
$$\frac{Area(S)}{2\pi}=2-2g-n-p+\sum_{i=1}^{n} \alpha_{i}+\sum_{j=1}^{m} (\beta_{j}-\frac{1}{2})$$
In particular, both terms of this equation are assumed to be positive.

\subsection{Monodromy}

For any loop $\gamma$ in a spherical surface $S$ punctured at its conical singularities and any base point $x \in \gamma$, a spherical chart around $x$ and the same spherical chart obtained after parallel transport along $\gamma$ are identified up to a rotation $\rho_{\gamma} \in SO(3)$.\newline

Just by looking at singular spheres $S_{\alpha}$ (see Theorem 1.3), it is clear that for a simple direct loop around a conical singularity of angle $2\alpha\pi$, the monodromy along this loop is the rotation of angle $[2\alpha\pi]\in \mathbb{R}/ 2\pi\mathbb{Z}$ whose axis passes through the conical singularity.\newline

For any chart $\phi$ covering a regular point $x$ of $S$, the monodromy group of $S$ at $(x,\phi)$ is subgroup of $SO(3)$ formed by rotations $\rho_{\gamma}$ for any loop $\gamma$ of $S$ punctured at its conical singularities. Spherical surfaces with trivial monodromy (and without boundary) are completely characterized. In particular, in genus zero, trivial monodromy is equivalent to the fact that conical singularities have an angle that is an integer multiple of $2\pi$.

\begin{thm}
A spherical surface $S$ with trivial monodromy and without boundary is a cover of the sphere ramified at the conical singularities. Their angles are of the form $2m\pi$ where $m$ is their (integer) ramification index. 
\end{thm}

Theorem 2.1 and Riemann-Hurwitz formula implies the following impossibility result.

\begin{cor}
The only spherical surface of genus zero without boundary and with a unique conical singularity is the standard sphere with a marked point.
\end{cor}

It also proves that there is only one monogon in spherical geometry.

\begin{cor}
The only spherical surface of genus zero with one boundary component formed by a unique geodesic arc and without any interior conical singularity is the half-sphere bounded by a great circle with a marked point on it.
\end{cor}

\begin{proof}
For any such monogon, we duplicate it and glue the two surfaces obtained on their boundary. We get a sphere without boundary with only one conical singularity. Corollary 2.2 implies that this conical singularity has angle equal to $2\pi$. We get in fact the standard sphere and the two monogons are just two half-spheres separated by a great circle with a marked point on it.
\end{proof}

The latter corollary then implies a technical result on geodesic trajectories in spherical polygons.

\begin{cor}
In a spherical surface of genus zero, without interior conical singularities and with a unique boundary component, there is no self-intersecting geodesic curve.
\end{cor}

\begin{proof}
Such a self-intersecting geodesic curve would contain a singular loop formed by a geodesic arc and a turning point. One connected complement of this singular loop would be contractible. Thus, it would be a monogon. Corollary 2.3 implies that such a monogon is just a half-sphere bounded by a great circle. Consequently, the singular loop is a closed geodesic and is thus not self-intersecting.
\end{proof}

\subsection{Geodesic arcs}

We prove that length minimizer of nontrivial homotopy class of topological arcs are broken geodesic arcs.

\begin{lem}
In any nontrivial homotopy class of topological arcs without self-intersection joining conical singularities (possibly the same), any arc that minimizes length functional is a broken geodesic arc: it is formed by several geodesic arcs joining conical singularities.
\end{lem}

\begin{proof}
Since the homotopy class is assumed to be nontrivial, the length of any arc of this class is bounded by below by injectivity radii of the conical singularities of the surface. Using a compacity argument, there is an arc $\alpha$ that minimizes length in this class. This implies that the arc is conjugated to great circles in the local spherical charts. For any regular point of the arc, the geodesic can be continued in both directions. Since it has finite length, in each direction the geodesics hits a conical singularity. Any regular point of the arc belongs to a geodesic arc. A compactness argument then shows that $\alpha$ is a broken geodesic arc formed by finitely many geodesic arcs joining conical singularities.
\end{proof}

\subsection{Singular locus}

\begin{defn} In a spherical surface $S$, we define the \textit{singular locus} $Sing(S)$ of $S$ as the union of the conical singularities and the boundary.
\end{defn}

The following two propositions show that in any spherical surface $S$, there is always enough disjoint geodesic arcs to decompose into the gluing of finitely many spherical polygons (i.e. spherical surfaces of genus zero with a connected singular locus).

\begin{prop} In any spherical surface $S$ whose singular locus $Sing(S)$ is not connected, there is a geodesic arc connecting two conical singularities belonging to distinct components of the singular locus.
\end{prop}

\begin{proof}
For any pair of conical singularities that do not belong to the same component of $Sing(S)$ there is an homotopy class of topological arcs without self-intersection that join them. We take the shortest topological arc $\alpha$ in this homotopy class. Following Lemma 2.5, $\alpha$ is a broken geodesic arc formed by several geodesic arcs joining conical singularities. At least one of these simple geodesic arcs joins two distinct components of $Sing(S)$.
\end{proof}

\begin{prop} In any spherical surface $S$ of genus $g>0$ and whose singular locus $Sing(S)$ is connected (either $S$ is closed with a unique conical singularity or it has a unique boundary component and no interior singularity), there is a geodesic arc distinct from the boundary and whose complement is connected.
\end{prop}

\begin{proof}
We proceed quite similarly as in the proof of Proposition 2.7. Since $g>0$, there is an homotopy class of topological arcs without self-intersection that join conical singularities and whose complement is connected. We take the shortest topological arc $\alpha$ in this homotopy class. It is a broken geodesic arc formed by several geodesic arcs joining conical singularities. The original topological arc is not zero homologous in $H_{1}(S; Sing(S))$. Thus, in the minimizing broken geodesic at least one of geodesic arcs is not zero homologous and hence does not cut the surfaces into two.
\end{proof}

\subsection{Spherical polygons and spherical projection}

A \textit{spherical polygon} is a spherical surface of genus zero with one boundary component formed by geodesic arcs and without interior conical singularities. In particular, they are not always subsets of the standard sphere. Spherical polygons have trivial monodromy. Therefore, their developing map induces a projection to the standard sphere (well-defined up to a rotation of $SO(3)$).

\begin{lem}
If $S$ is a spherical polygon whose spherical projection $Y$ is contained in an open half-sphere $H$, then the convex hull $Y'$ of $Y$ in $H$ is also a polygon contained in $H$ with at least three vertices. Their angles are strictly smaller than $\pi$. Besides, these vertices are projections of vertices of $S$ with the same angle.
\end{lem}

\begin{proof}
A subset $E$ of $H$ is convex if for every $x,y \in E$, the portion of great circle passing through $x$ and $y$ which is contained in $H$ is also contained in $E$. The convex hull $Y'$ of $Y$ is the union of any such portions of great circles in $H$ between points of $Y$. As such, it is also a subset of $H$. The spherical projection $Y$ of polygon $S$ is also a polygon. Its convex hull $Y'$ is also the convex hull of the vertices of $Y$ in the sphere. Therefore, $Y'$ is also a spherical polygon of $H$. Since it is convex, the angles of its vertices are strictly smaller than $\pi$. Monogons are half-spheres (see Corollary 2.3) and Digons have boundary arcs of length $\pi$ so $Y'$ has at least three vertices. Therefore, $Y$ has at least three vertices of angle smaller than $\pi$. These vertices cannot be the projection of regular point of $S$ nor regular points of its boundary. These vertices are the projection of vertices of $S$ with an angle strictly smaller than $\pi$.
\end{proof}

\subsection{Digons}

The following theorem is equivalent to the main result of \cite{Tr2}. It is the complete classification of spherical surfaces of genus zero with two conical singularities. They will be building blocks in our further constructions.

\begin{thm}
Any spherical surface of genus zero with two conical singularities and without boundary is an element of one of the two following families:\newline
(i) \textit{Foliated singular sphere} $S_{\alpha}$ : a topological sphere with two conical singularities with the same angle $\alpha>0$ foliated by a continuous family of disjoint geodesic arcs of length $\pi$;\newline
(ii) \textit{Pearl row} $P_{m,L}$ : $m \geq 1$ spheres glued cyclically along a slit of length $L$ with $0<L<\pi$. It is a cover of the standard sphere with two ramification points of index $m$ (angles $2m\pi$ and trivial monodromy).
\end{thm}

A \textit{foliated digon $D_{\alpha}$ of angle $2\alpha\pi$} is the spherical surface with boundary obtained by cutting $S_{\alpha}$ along one of its geodesic arcs connecting the two singularities. Since they are foliated by geodesic arcs of length $\pi$, their existence is the main obstruction to get a decomposition theorem as existence of a maximal system of disjoint geodesic arcs. If there is a foliated digon in the decomposition, we can always draw additional geodesic arcs.\newline

We prove that digons $D_{\alpha}$ are the only possible digons in spherical geometry.

\begin{prop}
Every spherical polygon with two vertices whose sum of angles is not an integer multiple of $2\pi$ is a digon $D_{\alpha}$.
\end{prop}

\begin{proof}
For any such digon with two singularities of angles $\theta$ and $\eta$, we duplicate it and get another digon with two singularities of angles $\theta$ and $\eta$. We glue sides on each other in order to get a sphere with two conical singularities of angle $\theta+\eta$. Theorem 2.10 then implies that the obtained sphere $S_{\theta+\eta}$. Cutting along two disjoint geodesic arcs between the two conical singularities in order to get isometric digons provides two digons isometric to $D_{\frac{\theta+\eta}{2}}$.
\end{proof}

\section{Irreducible polygons}

In this section, we characterize spherical polygons in which no additional geodesic arc can be drawn.

\begin{defn}
A spherical polygon $S$ is \textit{irreducible} if every geodesic arc of $S$ whose endpoints are conical singularities belongs to its boundary.
\end{defn}

\subsection{Half-spheres}

A first family of irreducible spherical polygons is formed by half-spheres. They highlight a clear difference with the case of flat surfaces with conical singularities. Here, for any number $n \geq 1$, there are irreducible polygons (i.e. without diagonals) with $n$ sides.

\begin{prop} In a half-sphere bounded by a great circle with at least one marked point and no pair of antipodal marked points, there is no geodesic arc joining singularities outside the boundary.
\end{prop}

\begin{proof}
In any spherical map, a geodesic arc is a part of a great circle. Between two points of the boundary that are not antipodal, the only great circle they belong to is the boundary itself.
\end{proof}

If two marked points of the boundary are antipodal, then the polygon is in fact foliated digon $D_{1/2}$ possibly with marked points on its sides. In particular, such a spherical polygon has infinitely many disjoint geodesic arcs.

\subsection{Characterization of irreducible spherical polygons}

The following lemma uses implicitly the idea of cut-locus to exhibit new geodesic arcs to decompose our spherical surfaces.

\begin{lem}
If there are two (regular or conical) points $x$ and $y$ in a spherical surface $S$ such that there is a geodesic arc $\gamma_{0}$ of length $\pi$ between $x$ and $y$, then $\gamma_{0}$ belongs to a $1$-parameter family of geodesic arcs of length $\pi$ between $x$ and $y$. This family is an embedding of a maximal foliated digon $D_{\alpha}$ into $S$. One of the three following proposition holds:\newline
(i) Surface $S$ is a singular sphere $S_{\alpha}$ or a foliated digon $D_{\alpha}$.\newline
(ii) One extremal element of the family contains a conical singularity while the other coincides with a boundary arc of $S$.\newline
(iii) Each of the two extremal elements of the family contains a conical singularity. If digon $D_{\alpha}$ is such that $\alpha<\frac{1}{2}$, then these two conical singularities are joined by a geodesic arc that belongs to the image of $D_{\alpha}$.\newline
Besides, if $D_{\alpha}$ is such that $\alpha \geq \frac{1}{2}$, then its image contains a half-sphere embedded in $S$ whose boundary is a (possibly broken) closed geodesic that contains both $x$ and $y$
\end{lem}

\begin{proof}
There is a spherical chart that contains entirely the interior points of $\gamma_{0}$. This geodesic arc appears as a meridian of the sphere. The neighborhood of this meridian contains other meridians that also correspond in $S$ to geodesic arcs of length $\pi$ between $x$ and $y$. Therefore, $\gamma_{0}$ belongs to a $1$-parameter family $\gamma_{t}$ of locally disjoint geodesic arcs of length $\pi$ (they are meridians) between $x$ and $y$.\newline
We consider a maximal interval for which two points of arcs of the family coincide and without any conical singularity other than $x$ and $y$. Three things can happen. First, two points may coincide in a point of tangency between two geodesics arcs of the family. Two geodesics cannot be tangent without being the same. This only happens if $S$ is in fact a singular sphere $S_{\alpha}$. Secondly, an extremal element of the family could be tangent to the boundary of $S$. Since the boundary of S is geodesic, using the same argument, we conclude that the extremal element coincides (at least partially) with a boundary arc of $S$. Finally, the extremal arc could be broken by a conical singularity.\newline
If there is one conical singularity on each extremal element of the family and $D_{\alpha}$ is such that $\alpha<\frac{1}{2}$, there is a spherical chart containing the digon and such that the two conical singularities belong to the same open half-sphere. Therefore, there is a portion in the great circle in the chart contained in $D_{\alpha}$ that joins the two singularities. This portion of great circle is embedded in $S$ as a geodesic arc.\newline
Finally, if we have $\alpha \geq \frac{1}{2}$, then digon $D_{1/2}$ (which is just the half-sphere) is embedded in $S$. Its boundary is a great circle passing through $x$ and $y$. There can be other conical singularities or portions of the boundary of $S$ if the boundary of the image of $D_{1/2}$ coincides with one of the two extremal elements of the maximal family of geodesic arcs.
\end{proof}

The following proposition states a first constraint on vertices of irreducible spherical polygons.

\begin{prop}
For any conical singularity $M$ of angle at least $\pi$ in an irreducible spherical polygon $S$ which is not a half-sphere, the length of any geodesic curve starting from $M$ is always strictly smaller than $\pi$.
\end{prop}

\begin{proof}
We assume there is such geodesic curve $\gamma$ of length at least $\pi$. Lemma 3.3 implies that $\gamma$ belongs to the embedding of a maximal digon $D_{\alpha}$ (a family of geodesic arcs of length $\pi$ with the same end). The two extremal arcs of the family either coincide partially with one of the boundary arcs of $S$ incident to $M$ or contain a singularity. Both of the two extremal arcs cannot coincide with boundary arcs of $S$ unless $S$ is itself a digon. Therefore, at least one of the extremal arc contains a singularity. By hypothesis, the geodesic arc joining $P_{i}$ to this singularity belongs to the boundary of $S$. Consequently, the maximal digon contains the whole angular sector around $M$ and its angle is at least $\pi$. Therefore, it contains a half-sphere bounded by a (possibly broken) closed geodesic. Since $S$ is irreducible, it should coincide with the half-sphere. We get the adequate contradiction.
\end{proof}

The latter proposition implies an impossibility result for irreducible digons.

\begin{cor}
Irreducible spherical polygon that are not half-spheres have at least three vertices of angle different from $\pi$.
\end{cor}

\begin{proof}
If a spherical polygon with marked points on its boundary is irreducible, then replacing the marked points by regular points we also get an irreducible spherical polygon. Therefore, it is enough to prove the proposition for polygons without marked points. We consider an irreducible spherical polygon $X$ with exactly two vertices $A$ and $B$ of angles $\alpha,\beta$ (and which is not a half-sphere). We duplicate $X$ and get another digon such that the boundary arcs from $A$ to $B$ and $B$ to $A$ are isometric. Gluing these two polygons provides a sphere $Y$ with at most two conical singularities $C$ and $D$ (both of angle $\alpha+\beta$). If $\alpha+\beta$ is not an integer multiple of $2\pi$, Proposition 2.11 already proved that $X$ is a foliated digon (so it is not irreducible). If $\alpha+\beta$ is an integer multiple of $2\pi$, then the angles of $Y$ are integer multiple of $2\pi$ and $Y$ is a ramified cover of the standard sphere. Besides, among $\alpha$ and $\beta$, at least one of them is at least $\pi$. Without loss of generality, we assume $\alpha \geq \pi$. Proposition 3.4 then implies that every geodesic curve starting from $A$ has a length strictly smaller than $\pi$. Therefore, the two boundary arcs of $X$ are strictly smaller than $\pi$. Since $Y$ is a ramified cover of the standard sphere, we get that the projection of the two boundary arcs is the same geodesic arc of length strictly smaller than $\pi$. Consequently, $\alpha$ and $\beta$ are themselves multiple integers of $2\pi$ and each version of $X$ has a surjective projection on the standard sphere. Therefore, some geodesic curves starting from $A$ or $B$ have a length which is at least $2\pi$. We get the adequate contradiction.\newline
Corollary 2.3 also implies that monogons are half-spheres.
\end{proof}

Like triangles in the flat case, irreducible spherical polygons have a well-defined notion of opposite side for a vertex.

\begin{lem}
In an irreducible spherical polygon, every geodesic curve starting from a conical singularity eventually crosses a boundary arc. Besides, every geodesic curve (distinct from the boundary arcs) starting from the same vertex crosses the same boundary arc. To any vertex is thus assigned an opposite side.
\end{lem}

\begin{proof}
Following Corollary 2.4, geodesic curves starting from a vertex are not self-intersecting. Looking at its spherical projection (see Subsection 2.4), it is also clear they are not accumulating on any point of the spherical polygon (there is no accumulating geodesic curve in the standard sphere). The only possibility is that they hit a conical singularity (by hypothesis, they would then belong to the boundary) or cross one boundary arc. Among the geodesic curves starting from the same conical singularity, the extremal arcs are the only curves that hit another conical singularity. Therefore, the others cross the same boundary arc (in order to switch to another arc, there should be an intermediate arc that hits the vertex between the two boundary arcs). To any vertex $P_{i}$ is thus assigned an opposite boundary arc.
\end{proof}

The following proposition proves another nontrivial constraint on vertices of irreducible spherical polygons.

\begin{prop}
For any irreducible spherical polygon $S$, if there is a singularity $P_{i}$ of angle smaller than $\pi$ such that the length of any geodesic curve starting from it is strictly smaller than $\pi$, then $S$ is a spherical triangle with three angles strictly smaller than $\pi$ (with possibly marked points on its sides).
\end{prop}

\begin{proof}
The conical singularities of the boundary of $S$ are $P_{1},\dots,P_{n}$. Their angles are $2\pi\beta_{1},\dots,2\pi\beta_{n}$.\newline
Geodesic curves starting from $P_{i}$ are disjoint from each other (otherwise there would be a digon with geodesic arcs smaller than $\pi$). The union of these geodesic curves is a triangle with geodesic boundary, isometric to a triangle of $D_{\beta_{i}}$ a vertex of which is one of the conical singularities of the digon. One side of the triangle is a portion of the geodesic arc opposite to $P_{i}$. The two others are geodesic arcs that, at least near $P_{i}$ coincide with the boundary arcs incident to $P_{i}$. They may contain conical singularities.\newline
Since $\beta_{i}<\frac{1}{2}$, if there is such conical singularities on both sides, the triangle is contained in an open half-sphere and the two conical singularities can clearly be joined by a geodesic arc, contradicting the hypothesis.\newline
If one side is a boundary arc while the other contains a conical singularity, then there is a diagonal (geodesic arc between two conical singularities) inside the triangle (which is still contained in a half-sphere).\newline
Finally, if both sides are boundary arcs, then $S$ is a spherical triangle with an angle strictly smaller than $\pi$ bounded by two sides of length strictly smaller than $\pi$.
\end{proof}

Propositions 3.4 and 3.7 allow to state a dichotomy result on vertices of irreducible spherical polygons that have not been classified yet.

\begin{cor}
We consider an irreducible spherical polygon $S$ which is not a half-sphere nor a spherical triangle whose angles are smaller than $\pi$. Then any vertex $P$ of $S$ belongs to one of the two following types:\newline
Type I: Angle is smaller than $\pi$ and there is a geodesic curve starting from $P$ of length at least $\pi$;\newline
Type II: Angle is at least $\pi$ and the length of every geodesic curve starting from $P$ is strictly smaller than $\pi$.
\end{cor}

We also have a precise result on the distribution of the vertices of the two types in the cyclic order.

\begin{prop}
Let $S$ be an irreducible spherical polygon with $n\geq 3$ singularities $P_{1},\dots,P_{n}$. If $S$ is not a half-sphere nor a spherical triangle whose angles are smaller than $\pi$, then there are exactly two consecutive vertices of type I. The $n-2$ other vertices are of type II. Besides, the boundary arc between the two vertices of type I is the opposite side of every vertex.
\end{prop}

\begin{proof}
We first suppose that every vertex of the polygon is of type II. Following Lemma 3.6, to any vertex $P_{i}$ is assigned an opposite side such that any geodesic curve starting from $P_{i}$ eventually crosses this boundary arc (and no other). These geodesic curves are disjoint and form a triangle $T_{i}$ (with geodesic boundary) in $S$. If there is a singularity $P_{i+1}$ in a side of $T_{i}$ incident to $P_{i}$, then the opposite sides of $P_{i}$ and $P_{i+1}$ are clearly the same. Therefore, every vertex that is not incident to the opposite side of $P_{i}$ has the same opposite side as $P_{i}$. Consequently, there is a vertex $P_{j}$ whose opposite side is $P_{j+1}P_{j+2}$. We consider a geodesic curve $P_{j}A$ with $A$ in $P_{j+1}P_{j+2}$ such that the angle between $P_{j}A$ and  $P_{j}P_{j+1}$ is strictly smaller than $\pi$. Since geodesic curves starting from $P_{j}$ have a length smaller than $\pi$, triangle $P_{j}P_{j+1}A$ has two sides of length smaller than $\pi$ bounding an angle also smaller than $\pi$. Therefore, this spherical triangle has all of its angles smaller than $\pi$ and $P_{j+1}$ is not of type II. Therefore, there always exists at least one vertex of type I.\newline
Let $M$ be a vertex of type I. There is a geodesic curve $\gamma$ of length $\pi$ that belongs to a family of geodesic curves forming a maximal foliated digon whose angle is that of $M$. The extremal arcs of the family cannot both coincide with boundary arcs ($S$ would be a digon). And there cannot be conical singularities in both of the extremal arcs (otherwise they could be joined by an additional geodesic curve inside the digon). Therefore, one extremal arc coincides with a boundary arc (whose length is thus at least $\pi$) so its vertex other than $M$ is also of type I). The other extremal arc contains a conical singularity whose angle is thus at least $\pi$. Therefore, it is of type II. Vertices of type I form pairs of consecutive vertices whose neighbors are of type II.\newline
Any vertex of type II belongs to a maximal chain of consecutive vertex of type II (in the cyclic order of the polygon) bounded by vertices of type I. The vertices of this chain have automatically the same opposite side. Besides, for an extremal vertex $P_{i}$ of the chain such that $P_{i+1}$ is of type I, it is easy to check that the opposite side of $P_{i}$ is $P_{i+1}P_{i+2}$. As a consequence, the boundary sides bounding the chain of vertices of type II are the same geodesic arc. There is only one maximal chain of vertices of type II and there is only one pair of vertices of type I.\newline
It still needs to be proven that the opposite side of the vertices of type I is the same side as the others. Since they are of type I, geodesic curves starting from them fill a maximal foliated digon (see the proof of Proposition 3.4) such that one extremal arc contains a conical singularity (its neighbor of type II) and the other coincides with the boundary arc between the two vertices of type I (which is thus of length at least $\pi$). Every geodesic curves starting from them finally cross this boundary arc (at the other end of the foliated digon).
\end{proof}

One part of the latter proof can also be deduced from a more general statement, called Drop lemma (see Lemma 8.2 in \cite{PP}). It follows from this lemma that a spherical polygon with one angle strictly smaller than $\pi$ and all other angles strictly larger than $\pi$ always contains a half-sphere.\newline

Spherical projection (see Subsection 2.4) highlights the specific properties of irreducible spherical polygons.

\begin{prop}
The spherical projection of an irreducible polygon $S$ is contained in a half-sphere. If $S$ is not a half-sphere nor a spherical triangle with angles smaller than $\pi$, the great circle bounding this half-sphere extends the projection of the boundary side $A_{1}A_{2}$ connecting the two vertices of type I in $S$. Besides, the length of $A_{1}A_{2}$ is strictly larger than $\pi$.
\end{prop}

\begin{proof}
If $S$ is a half-sphere or a triangle with angles smaller than $\pi$, its spherical projection is clearly contained in a half-sphere. Outside these two cases, Proposition 3.9 implies that $S$ contains two consecutive vertices $A_{1},A_{2}$ of type I while the other vertices are of type II. Following Corollary 3.8 and Proposition 3.9, every geodesic curve starting from a vertex of type II is of length strictly smaller than $\pi$ and eventually crosses $A_{1}A_{2}$ (always with the same orientation). Therefore, the spherical projection of each vertex of type II belongs to the same half-sphere $H^{+}$ bounded by the great circle extending the projection of $A_{1}A_{2}$. Projections of boundary arcs also belong to this half-sphere $H^{+}$  (they are of length strictly smaller than $\pi$). Therefore, if there are points of $S$ whose spherical projection belong to the complement half-sphere $H^{-}$, every point of $H^{-}$ is the projection of a point of $S$. Preimages of $H^{-}$ in $S$ are half-spheres embedded in $S$. This would contradict either the fact of $S$ is irreducible or the fact that it is not a half-sphere itself. Therefore, the spherical projection of $S$ is contained in $H^{+}$.
If the length of $A_{1}A_{2}$ is exactly $\pi$, this geodesic arc belongs to a $1$-parameter family of disjoint geodesic arcs between $A_{1}$ and $A_{2}$ (forming a digon). This would contradict the fact that $S$ is irreducible. Finally, if the length of $A_{1}A_{2}$ is strictly smaller than $\pi$, then the maximal digon formed by geodesic curves starting from $A_{1}$ (we should remember that $A_{1}$ is of type I so has a geodesic curve of length at least $\pi$) has a conical singularity on each of its extremal arcs. This would imply existence of a geodesic arc disjoint from the boundary between these two conical singularities.
\end{proof}

\subsection{Half-spherical concave polygons}

\begin{defn}
A spherical polygon $S$ is said to be a \textit{half-spherical concave polygon} if it satisfies the following properties:\newline
(i) $S$ has exactly two consecutive vertices $A_{1},A_{2}$ with angles strictly smaller than $\pi$;\newline
(ii) The length of $A_{1},A_{2}$ is strictly bigger than $\pi$;\newline
(iii) The spherical projection of $S$ is contained in the closed half-sphere $H^{+}$ bounded by the great circle extending the projection of $A_{1}A_{2}$;\newline
(iv) $S$ has an arbitrary number (at least one) of vertices of angles at least $\pi$.
\end{defn}

Half-spherical concave polygons (see Figure 1) are the last family of irreducible polygon we need to prove our decomposition theorem.

\begin{figure}
\includegraphics[scale=0.3]{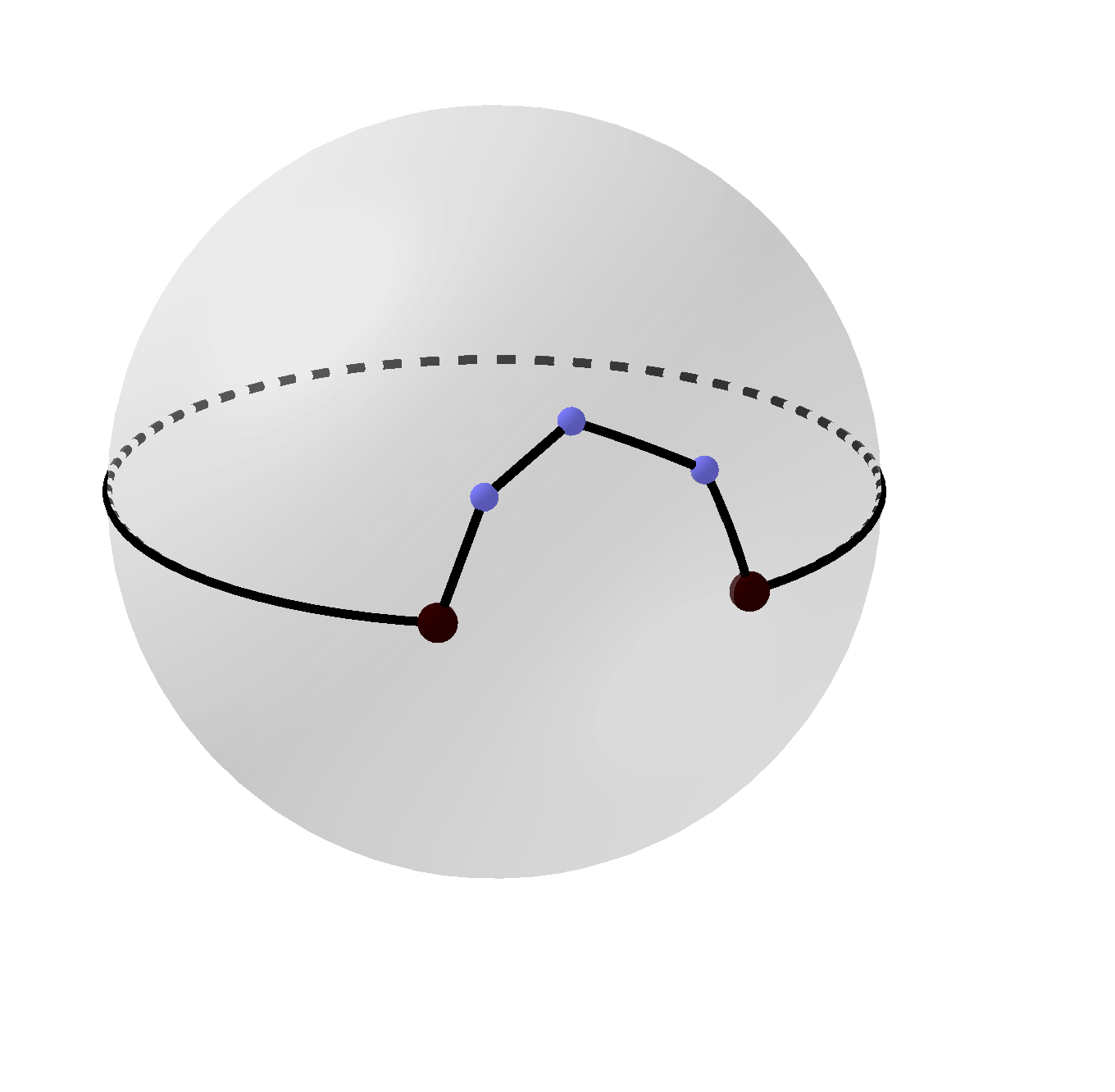}
\caption{The subset of the sphere which is above the polygonal curve is a half-spherical concave polygon.} 
\end{figure}

\begin{prop}
Half-spherical concave polygons are irreducible.
\end{prop}

\begin{proof}
We first note that the spherical projection of the vertices of angles strictly bigger than $\pi$ cannot belong to the boundary of $H^{+}$. This implies that boundary arcs between them also belong to the interior of $H^{+}$. Consequently, it is also true for marked points.\newline
We list each kind of geodesic arc distinct from the boundary and prove why there cannot exist.\newline
If there were a closed geodesic arc attached to one of the vertices (or the marked points) of $S$, it would cut out a monogon which would be a half-sphere (see Corollary 2.3). The projection of this half-sphere would coincide with $H^{+}$. The vertex incident to this closed geodesic arc should be $A_{1}$ or $A_{2}$. This would grant them an angle of $\pi$.\newline
If there were a geodesic arc between $A_{1}$ and $A_{2}$, then this arc would cut out a digon whose angles are strictly smaller than $\pi$. Thus, following Proposition 2.11 it would be a foliated digon and the length of $A_{1},A_{2}$ would be $\pi$.\newline 
If there were a geodesic arc between two vertices different from $A_{1}$ and $A_{2}$, then it would cut out $S$ into two components. The component that does not contain $A_{1}$ and $A_{2}$ is a spherical polygon whose spherical projection is contained in an open half-sphere. Lemma 2.9 implies that it should have at least three vertices with angles strictly smaller than $\pi$. By construction it has at most two such vertices (the ends of the geodesic arc since the other vertices have angle at least $\pi$).\newline
Finally, if there were a geodesic arc between $A_{1}$ (without loss of generality) and a vertex (or marked point) $B$ different from $A_{2}$, then it would cut out $S$ into two components. Let $T$ be the component that does not contain $A_{2}$. The spherical projection $Y$ of $T$ is a polygon of the closed half-sphere $H^{+}$. Since it has exactly one vertex on the great circle, $Y$ is contained in a slightly different half-sphere. Therefore, Lemma 2.9 can be used and implies that $T$ has at least three vertices with angles strictly smaller than $\pi$. By construction $T$ has at most two such vertices.\newline
Every geodesic arc of a half-spherical polygon is a boundary arc.
\end{proof}

Finally, we are able to state a complete characterization of irreducible spherical polygons.

\begin{thm}
Irreducible spherical polygons are the elements of the three following families:\newline
(i) Spherical triangles with angles strictly smaller than $\pi$;\newline
(ii) Half-spheres bounded by a great circle with marked points (with no pair of antipodal marked points in the boundary);\newline
(iii) Half-spherical concave polygons.
\end{thm}

\begin{proof}
Proposition 3.2 and 3.12 proves that polygons of families (ii) and (iii) are indeed irreducible. If a triangle with angles strictly smaller than $\pi$ contained a geodesic arc disjoint from the boundary, then it would cut out a digon whose angles would be strictly smaller than $\pi$. Proposition 2.11 would imply that it is a foliated digon and the length of one side of the triangle would be $\pi$.\newline
Then, Propositions 3.9 and 3.10 prove that an irreducible spherical polygon that is not a triangle nor a half-sphere satisfies the four conditions of Definition 3.11.
\end{proof}

The main theorem of the paper follows from Theorem 3.13 and some cuttings.

\begin{proof}[Proof of Theorem 1.3]
We consider a spherical surface $S$. If its singular locus is not connected, Proposition 2.7 implies existence of a geodesic arc connecting two distinct components of $Sing(S)$. The complement of this arc in $S$ is a new spherical surface whose singular locus has one connected component less. Iterating this process, we get a spherical surface $S'$ with a connected singular locus. If $S'$ has nonzero genus, Proposition 2.8 implies existence of a geodesic arc in $S'$ such that its complement has lower genus (and a singular locus which is still connected). Iterating this process, we get a spherical surface $S''$ of genus zero with a connected singular locus.\newline
If $Sing(S'')$ is just a conical singularity, then Corollary 2.2 implies that $S''$ is the standard sphere with a marked point. Cutting along a great circle passing through the marked point decomposes $S''$ into two half-spheres with a marked point on their boundary.\newline
Otherwise $S''$ is a spherical polygon (surface of genus zero with a connected boundary and without interior conical singularities). If $S''$ has only one side, it is a half-sphere with one marked point on its boundary (see Corollary 2.3). If $S''$ contains a geodesic arc whose ends are the same conical singularity, its complement has two connected components. One is a half-sphere with a marked point on its boundary (see Corollary 2.3). The other is a spherical polygon with one side more. However, its area has decreased by $2\pi$. Therefore, after finitely many such steps (the area of $S''$ gives a bound on the number of these steps), we get a spherical polygon $S'''$ with at least two sides and without any closed geodesic arc at any of its vertices.\newline
If there is a geodesic arc in $S'''$ between two nonconsecutive vertices, its complement has two connected components. These components have a strictly smaller number of sides than $S'''$. After finitely many steps, $S'''$ is decomposed into finitely many spherical polygons for which the only geodesic arcs between two conical singularities are between two consecutive vertices.\newline
Let $T$ by any such spherical polygon. If there is a geodesic arc between two consecutive vertices (and which is distinct from the boundary arc between them), then it cuts out a digon. Following Proposition 2.11, either the total angle in this digon is at least $2\pi$ or it is a digon foliated by a $1$-parameter family of disjoint geodesic arcs of length $\pi$. In the latter case, this family can be continued in $T$ until it hits another boundary arc (in this case, $T$ is a digon) or some conical singularities. Removing the geodesic arcs forming the boundary of this maximal digon cuts out $T$ into digons and some spherical polygons with a lower number of sides. Foliated digons can be arbitrarily cut into smaller foliated digons. We thus assume that foliated digons in our decomposition have two vertices with angles different from $2\pi$. If such a digon has marked points on both sides (between the two special vertices), we can cut it in order to get digons with marked points on only one side. Since this process always reduces the number of sides of spherical polygons, we finally get spherical polygons such that the length of any geodesic arc between conical singularities is not $\pi$. For such polygons, if there is a geodesic arc between two consecutive vertices, then it cuts a digon of total angle at least $2\pi$. Thus, after cutting finitely many such digons, we get a finite number of spherical polygons without any geodesic arc between the conical singularities and that would be distinct from the boundary.\newline
Irreducible spherical polygons have already been classified in Theorem 3.13.
\end{proof}

There is a natural question about combinatorial moves between geometric decompositions. Is it true that any pair of geometric decompositions of the same surface can be joined by a chain of moves that change one edge at a time. In other words, is the flip graph connected ?

\section{Core of a spherical surface}

\subsection{Motivation}

In the theory of translation surfaces, the geometric structure outside singularites is defined by local charts to $\mathbb{C}$ that are anti-derivatives of a holomorphic or meromorphic $1$-form (zeroes of the differentials are the conical singularities). The distinction between these two classes of differentials is crucial since a translation surface has finite area if the corresponding differential is holomorphic while any neighborhood of a pole has infinite area.\newline

This distinction seems not to have a counterpart for spherical surfaces (they always have finite area). The following examples might indicate that the situation is more complicated.

\begin{ex}
Let $A \subset \mathbb{C}$ be a parallelogram of the complex plane. If the sides of $A$ are identified, we get a torus with a translation structure of finite area (corresponding to the constant $1$-form on some elliptic curve). On the opposite, the complex plane $\mathbb{C} \setminus A$ with parallelogram $A$ removed and identified sides is a topological torus with a translation structure of infinite area. It corresponds to a meromorphic $1$-form with a pole of order $2$ and a zero of order $2$ (conical singularity of angle $6\pi$).
\end{ex}

The two constructions have very explicit counterparts in the real of tori with constant positive curvature and one conical singularity.

\begin{ex}
Let $A \subset S$ be a spherical quadrilateral (included in an open half-sphere) with two pairs of opposite sides of the same length. If the sides of $A$ are identified, we get a torus of constant positive curvature with a conical singularity. On the opposite, the complement of $A$ in the sphere with opposite sides identified is also a torus of constant positive curvature with one conical singularity.
\end{ex}

The two spherical surfaces constructed here are geometrically very different and it should be reflected by some invariant.\newline

For flat surfaces, the adequate notion is the core (see Section 4 in \cite{Ta1}) which is defined as the convex hull of the conical singularities for the underlying metric structure. The core is a connected compact subset of finite area of the surface, cut out by geodesic segments. Besides, its complement is a disjoint union of topological disks (one for each pole if the translation structure is induced by a meromorphic $1$-form).\newline

Such a definition would be trivial for spherical surfaces. We could proceed differently by defining the complement of the core.

\subsection{The core and the soul}

In this section, we assume spherical surfaces are closed (i.e. without boundary).

\begin{defn}
We denote by $H$ an open half-sphere of the standard sphere. Let $E(S)$ be the set of isometric embeddings of $H$ into a spherical surface $S$. For $U=\bigcup\limits_{\phi \in E(S)} \phi(H)$. The \textit{core} $\mathcal{C}(S)$ is the complement of $U$ in $S$.
\end{defn}

In Example 4.2, the core of the first spherical surface is the full surface. The core of the second one is just the union of the two geodesic arcs.\newline

Since the area of half-spheres is $2\pi$, $U$ is formed by finitely many connected components we call \textit{exterior domains}. In order to get a clear description of them, we introduce the \textit{soul} of an exterior domain.

\begin{defn}
We denote by $N \in H$ the uppermost point of $H$ (thought as the North Pole if $H$ is the northern hemisphere). Let $E(D)$ be the set of isometric embeddings of $H$ into an exterior domain $D$ of a spherical surface $S$. The soul of exterior domain $D$ is $Soul(D)=\bigcup\limits_{\phi \in E(D)} \phi(\{N\})$.
\end{defn}

Singular spheres $S_{\alpha}$ (see Subsection 2.5) are very specific. Their core is just the union of the two conical singularities. The complement of the core is connected. It is the unique exterior domain. It is not contractible and its soul is a closed geodesic (the equator of the singular sphere).\newline
Outside this exceptional case, we will prove that the core and the exterior domains satisfy the properties we expect: the core should be connected and exterior domains should be contractible.

\begin{lem}
For any exterior domain $D$ in a spherical surface $S$ which is not a singular sphere $S_{\alpha}$, $Soul(D)$ is either:\newline
(i) an isolated point if $D$ is a half-sphere bounded by a least three marked points and intervals of length strictly smaller than $\pi$ between them;\newline
(ii) a geodesic segment of length $\alpha$ if $D$ is foliated digon of angle $\alpha+\pi$ with at least one marked point on each of its sides;\newline
(iii) a spherical polygon whose angles are strictly smaller than $\pi$ otherwise.\newline
In particular, $Soul(D)$ is always connected.
\end{lem}

\begin{proof}
Any isometric embedding $\phi$ of $H$ into $D$ is completely characterized by $\phi(\{N\}) \in Soul(D)$. The soul of an exterior domain $D$ is always connected. Indeed, if two embedded half-spheres are not disjoint, their intersection is formed by a pencil of great circles that provide a geodesic segment inside $Soul(D)$ between the images of their uppermost points. The local geometry of $Soul(D)$ at each of this point depends on the boundary of the embedded half-sphere centered on this point.\newline
If $\phi(H)$ is a half-sphere bounded by a closed geodesic passing only through regular points, then there is a $2$-parameters local deformation (bending and turning) of the embedding. There is a open neighborhood of $\phi(H)$ that is completely covered by isometrically embedded half-spheres. Consequently, there exists an open neighborhood of $\phi(\{N\})$ that is completely contained in $Soul(D)$. $\phi(\{N\})$ is an interior point of $V$.\newline
If $\phi(H)$ is a half-sphere bounded by a closed geodesic with singularities so that the longest arc between two consecutive singularities is strictly smaller than $\pi$, then $\phi(\{N\})$ is isolated in $Soul(D)$ (there is no nontrivial deformation).\newline
If $\phi(H)$ is a half-sphere bounded by a closed geodesic with singularities such that there is at least on arc between two consecutive singularities whose angle is exactly $\pi$, then the half-sphere is in fact foliated digon $D_{1/2}$ whose $1$-parameter of geodesic arcs can be continued until it hits a conical singularity (the family cannot be circular since we excluded the case of foliated singular spheres). In this case, $D$ is a foliated digon of angle $\alpha+\pi$ with $\alpha >0$. Its soul is an interval of length $\alpha$ in the geodesic curve formed by the midpoints of the geodesic arcs of the digon.\newline
If $\phi(H)$ is a half-sphere bounded by a closed geodesic with singularities such that the longest arc between two consecutive singularities is of length $\beta$ with $\pi < \beta \leq 2\pi$, then the local geometry of $Soul(D)$ around $\phi(\{N\})$ is that of a corner of angle $\beta-\pi$ between two geodesic arcs. In particular, if there is only one singularity in the boundary of $\phi(H)$, then the longest arc is of length $2\pi$ and the local geometry of $Soul(D)$ around $\phi(\{N\})$ is that of a corner of angle $\pi$ (a half-sphere).\newline
In the following, we assume $D$ is not a half-sphere or a foliated digon. Therefore, the local geometry of $Soul(D)$ around each of its points is either that of an interior point, a boundary point or a vertex point of a spherical polygon. Since $S$ is compact, $Soul(D)$ is clearly a domain of $D$ cut out by finitely many geodesic arcs. Its angles are at most $\pi$. Gauss-Bonnet formula then implies that Euler characteristic of $Soul(D)$ is either $0$, $1$ or $2$. In the latter case $V$ is a topological sphere without boundary so there is no conical singularity at all in $S$. If Euler characteristic of $Soul(D)$ is zero, then $Soul(S)$ is an annulus bounded by two closed geodesics. Gauss-Bonnet formula implies such an annulus does not exist in spherical geometry. Otherwise, the Euler characteristic of $Soul(S)$ is one and it is a spherical polygon.
\end{proof}

There is a complete correspondance between the shape of the soul and the shape of an exterior domain.

\begin{thm}
In any spherical surface $S$ which is not a singular sphere $S_{\alpha}$, every exterior domain is either a foliated digon of angle at least $\pi$ or the complement in the sphere of a (possibly degenerate) convex polygon contained in a closed half-sphere.\newline
Besides, the core $\mathcal{C}(S)$ is a connected subset of $S$ that contains every conical singularity of $S$ and whose boundary is formed by geodesic arcs between the singularities. The length of its sides is strictly smaller than $\pi$.
\end{thm}

\begin{proof}
For any exterior domain $D$ of $S$, the cases where $Soul(D)$ is an isolated point or a geodesic segment are already settled in Lemma 4.5. In the following, we assume $Soul(D)$ is a (nondegenerate) spherical polygon.\newline
Let $x$ be an element of the boundary of $D$. As such it belongs to the boundary of an embedded sphere $\phi(H) \subset D$. Clearly, $\phi(\{N\})$ should also belong to the boundary of $Soul(D)$. If $\phi(\{N\})$ is a vertex of angle $\beta<\pi$ in $Soul(D)$, the boundary of $\phi(H)$ is a great circle with conical singularities such that the length of the longest arc between two singularities is $\pi+\beta$. Element $x$ of the boundary of $D$ should belong to the complement arc of length $\pi-\beta$ (otherwise we could deform the embedding $\phi$ to have $x$ as an interior point of $D$). Elements of this complement arc of length $\pi-\beta$ clearly belong to the boundary of $D$. They form a chain of geodesic arcs between conical singularities.\newline
If $\phi(\{N\})$ belongs to the boundary of $Soul(D)$ but is not a vertex, then the boundary of $\phi(H)$ is a great circle with only one conical singularity (of angle at least $\pi$ inside $D$). We can find a 1-parameter family of half-spheres embeddings turning around this singularity until the boundary of one of these half-spheres hits another singularity. Therefore, both the boundary of $D$ and that of $Soul(D)$ are closed polygonal curves. Vertices of the boundary of $D$ are conical singularities and its sides are geodesic arcs. There is a duality between these two polygonal curves: An angle $\beta$ in $Soul(D)$ is dual to a side of length $\pi-\beta$ in the boundary of $D$. A side of length $L$ in $Soul(D)$ corresponds to a singularity with an angle of $\pi+L$ in $D$.\newline
Since elements of $D$ are always regular points of $S$, every conical singularity belongs to the core $\mathcal{C}(S)$. Exterior domains are spherical polygons (topological disks) so they are separated from each other by the core. As such, if we make one puncture in each exterior domain, the punctured surface retracts on the core which is thus connected.\newline
It remains to prove that $D$ is the complement in the sphere of a (possibly degenerate) polygon contained in a closed half-sphere.\newline
Following Lemma 4.5, $Soul(D)$ is a nondegenerate spherical polygon whose vertices have angles strictly smaller than $\pi$. If there is a geodesic arc $\gamma$ of length $\pi$ in $Soul(D)$ between two (possibly regular) points $x$ and $y$, then $\gamma$ belong to an embedded maximal foliated digon. Since $Soul(D)$ is not itself a digon, then the two extremal arcs of the digon meet the boundary of $Soul(D)$. However, two geodesics that are tangent to each other are in fact the same. Besides, interior angles of $Soul(D)$ are strictly smaller than $\pi$ so they cannot belong to an extremal arc. Therefore, existence of such an embedded foliated digon implies a contradiction. The diameter of $Soul(D)$ is thus strictly smaller than $\pi$. Together with the fact that interior angles are strictly smaller than $\pi$, this implies that $Soul(D)$ is a convex spherical polygon contained in a half-sphere.\newline
Correspondance between angles and sides of $Soul(D)$ and $D$ then implies that $D$ can be embedded into a sphere in such a way the complement of $D$ in the sphere is the classical dual polygon of $Soul(D)$. Since $D$ contains a half-sphere (by hypothesis), it is the complement of a spherical polygon contained in a half-sphere. It is then easy to show that the length of its sides is strictly smaller than $\pi$ and the angles are at least $\pi$ and strictly smaller than $2\pi$.
\end{proof}

In particular, the class of complements in the sphere of a possibly degenerate polygon contained in a closed half-sphere contains:\newline
(i) Half-spheres bounded by at least three marked points and intervals of length strictly smaller than $\pi$ between them;\newline
(ii) Spheres with a slit of length strictly smaller than $\pi$.\newline
It is clear that the latter family is the only possible shape for exterior domains with exactly two boundary arcs (and there is no exterior domain with only one boundary arc).

\begin{cor}
If $S$ is a spherical surface distinct from a foliated singular sphere $S_{\alpha}$, each exterior domain in $S$ is a spherical polygon bounded by at least two arcs.
\end{cor}

\begin{proof}
Since exterior domains are topological disks (Theorem 4.6), if they are bounded by only one arc, they are monogons. We consider such an exterior domain. Following Corollary 2.3, a monogon is a half-sphere with only one marked point $A$ on a great circle $\gamma$. If we rotate slightly the great circle $\gamma$ around the axis passing through $A$ in the half-sphere and get, in a neighborhood of $\gamma$, another great circle bounding a half-sphere. The two half-spheres have nontrivial intersection so the exterior domain is bigger than a half-sphere.
\end{proof}

\begin{cor}
In any closed spherical surface, an exterior domain bounded by two arcs is a sphere with a slit of length $L<\pi$. Its total angle and its area are both equal to $4\pi$.
\end{cor}

\begin{proof}
If an exterior domain $D$ is a digon, then it cannot be a foliated digon because in this case, Lemma 4.5 implies that it has at least four boundary arcs (there is at least one conical singularity on each extremal arc of the digon otherwise the continuous family of geodesic arcs can be continued). $Soul(D)$ is also a digon. Theorem 4.6 implies that $D$ is the complement in the sphere of a polygon with exactly two boundary arcs. The only possibility is that $D$ is the sphere with a slit. The two vertices have an angle of $2\pi$ and the area of $D$ is that of the sphere.
\end{proof}

\subsection{Geometric decompositions of the core and exterior domains}

Theorem 1.3 proves that any spherical surface $S$ can be decomposed into pieces of simple shape. However, there could be a lot of such decompositions. It could be reasonable to focus on geometric decompositions in which the core is a union of pieces of the decomposition. 

\begin{cor}
For a (closed) spherical surface $S$ (which is not a foliated singular sphere), there is a geometric decomposition of the interior of $Core(S)$ into spherical triangles with angles strictly smaller than $\pi$.
\end{cor}

\begin{proof}
Most of propositions follow from Theorems 1.3 and 4.6. The interior of $core(S)$ is a subsurface of $S$ cut out by geodesic arcs between conical singularities (it is a spherical surface with boundary). By definition, there is no half-sphere in the core and any spherical triangle whose vertices (conical singularities) belong to the same half-sphere in the spherical projection cannot belong to embedding of a half-sphere. If there is a foliated digon inside an exterior domain, then its family of geodesic arcs can be continued til the boundary of the exterior domain. In this case, the exterior domain is itself a foliated digon.\newline
Foliated digons inside the core can also be continued until extremal arcs hit some conical singularity (the family of geodesic arc cannot be cyclic because in this case the whole surface would be a foliated singular sphere).\newline
Half-spherical concave polygons are the only polygons of the decomposition that have boundary arcs of length strictly bigger than $\pi$. Thus, they form pairs glued along their longest side. For each pair, we remove the edge $A_{1}A_{2}$ that separates them an get a spherical polygon $P$. We take a part of edge $A_{1}A_{2}$ of length $\pi$ starting from $A_{1}$. This geodesic arc of length $\pi$ belongs to a maximal embedded digon whose extremal arcs contain strictly the incident boundary arcs of $A_{1}$ (extremal arcs cannot hit other conical singularities because it would imply existence of a diagonal which is forbidden by Proposition 3.12). If the angle obtained at a vertex $A_{1}$ is at least $\pi$, then there is an embedded half-sphere in the digon (which is a clear contradiction since there polygons belong to the core). If the angle is strictly smaller than $\pi$, then the two neighbors of $A_{1}$ can be joined by a geodesic arc disjoint from the boundary. This arc cuts out a spherical triangle with angles strictly smaller than $\pi$. The complement of this triangle in $P$ is a polygon with a vertex less than $P$. Theorem 1.3 can be applied to this new polygon and iterating the process, we finally get a complete decomposition into triangles.
\end{proof}

The number of triangles in such decompositions of the core is controlled by the number of exterior domains of the number of their sides. This provides also a topological bound on the total angle of the core.

\begin{cor}
For a (closed) spherical surface $S$ (which is not a foliated singular sphere) of genus $g$ with $n$ conical singularities, $p$ domains of poles with $b_{1},\dots,b_{p} \geq 2$ boundary arcs, the interior of $\mathcal{C}(S)$ decomposes into $4g-4+2n+2p-\sum_{i=1}^{p} b_{i}$ spherical triangles.\newline
In particular, the total angle of the core is strictly smaller than $(12g-12+6n)\pi$. Its area is strictly smaller than $(8g-8+4n)\pi$.
\end{cor}

\begin{proof}
Corollary 4.9 implies existence of a geometric decomposition of the core into $t$ triangles. We get an embedded graph into $S$ with $n$ vertices, $p+t$ contractible faces and $\frac{1}{2}(3t+\sum_{i=1}^{p} b_{i})$ edges. Therefore, Euler characteristic provides the value of $t$. Every spherical triangle of the core has a total angle strictly smaller than $3\pi$ (and an area strictly smaller than $2\pi$). Besides, Corollary 4.7 implies that $2p-\sum_{i=1}^{p} b_{i} \leq 0$.
\end{proof}

Bounds on the area provide an estimation of the number of exterior domains.

\begin{cor}
For a (closed) spherical surface $S$ (distinct from a singular sphere) of genus $g$ with $n$ conical singularities of angles $2\pi\alpha_{1},\dots,2\pi\alpha_{n}$. If there $S$ has $p$ domains of poles with $b_{1},\dots,b_{p}$ boundary arcs none of which being a foliated digon, then we have
$p \geq 3-3g-\frac{3n}{2}+\frac{1}{2}\sum_{i=1}^{n} \alpha_{i}$.\newline
Even if we allow foliated digons as exterior domains, we also have the following upper bound $p \leq 2-2g-n+\sum_{i=1}^{n} \alpha_{i}$.
\end{cor}

\begin{proof}
Gauss-Bonnet formula states that $Area(S)=(4-4g-2n+2\sum_{i=1}^{n} \alpha_{i})\pi$. Corollary 4.10 proves that the total area of the core is at most $(8g-8+4n-2\sum_{j=1}^{p} (b_{j}-2))\pi$. Consequently, the total area of exterior domains is at least:
$$(2\sum_{i=1}^{n} \alpha_{i}+12-12g-6n+2\sum_{j=1}^{p} (b_{j}-2))\pi.$$
Theorem 4.6 proves that the area of any exterior domain which is not a foliated digon is at most $4\pi$. Therefore we have $p \geq 3-3g-\frac{3n}{2}+\frac{1}{2}\sum_{i=1}^{n} \alpha_{i}+\frac{1}{2}\sum_{j=1}^{p} (b_{j}-2)$. Since every exterior domain has at least two boundary arcs, we get the lower bound.\newline
Every exterior domain contains a half-sphere so its area is at least $2\pi$. This induces the upper bound on the number of exterior domains.
\end{proof}

As a consequence of these estimates, we can deduce that for a given $n$ and $g$, every spherical surface with a large enough total conical angle, at least one exterior domain is just a sphere with a slit (a pearl). Therefore, spherical surfaces with a total conical angle exceeding the bound are easily reduced to spherical surfaces that satisfy the bound.

\begin{thm}
For a (closed) spherical surface $S$ (distinct from a singular sphere) of genus $g$ with $n$ conical singularities of angles $2\pi\alpha_{1},\dots,2\pi\alpha_{n}$. We assume no exterior domain is a foliated digon. If the total conical angle $2\pi\sum_{i=1}^{n} \alpha_{i}$ is at least $(10g-10+5n)2\pi$, then at least one exterior domain is a sphere with a slit.
\end{thm}

\begin{proof}
We assume $S$ has $p$ domains of poles with $b_{1},\dots,b_{p} \geq 3$ boundary arcs. Spherical surface $S$ decompose into $p$ exterior domains of area strictly smaller than $4\pi$ and a core formed by $4g-4+2n+2p-\sum_{i=1}^{p} b_{i}$ spherical triangles of area strictly smaller than $2\pi$.\newline
The maximal area is obtained with a degenerate core and exterior domains with exactly three boundary arcs. In this case, $p=4g-4+2n$ and the total area is strictly smaller than $(16g-16+8n)\pi$. Gauss-Bonnet formula states that the total area of the surface is $(4-4g-2n+2\sum_{i=1}^{n} \alpha_{i})\pi$. Therefore, we have $\sum_{i=1}^{n} \alpha_{i} < 10g-10+5n$.
\end{proof}

The bound $(10g-10+5n)2\pi=2\pi(-5\chi(\dot{S}))$ (where $\dot{S}$ is the surface punctured at the singularities) can be proven to be optimal. Indeed, following a construction suggested by the anonymous referee, we consider a spherical surface $S$ formed by the gluing of $2m$ triangles of angles $2\pi(\frac{5}{6}-\epsilon)$ (these triangles are complements to tiny equilateral triangles in the standard sphere). These surface clearly does not contain pearls and foliated digons. At the same time, the sum of conical angles in $S$ is $6m.2\pi(\frac{5}{6}-\epsilon)$ that can be made arbitrarily close to $5m.2\pi=2\pi.(-5\chi(\dot{S}))$.\newline

Our intuition here is that exterior domains should play the same role in spherical surfaces as poles in translation surfaces corresponding to meromorphic differentials. Some spherical surfaces seems to be the counterparts of holomorphic $1$-forms (those for which the core is the full surface). While others are the counterparts of meromorphic $1$-forms (those that have embedded half-spheres).

\subsection{Core and discriminant}

We define a notion of discriminant in the moduli spaces of spherical surfaces in such a way that outside the discriminant, the topological pair formed by the core in the surface is structurally stable.\newline

\begin{defn}
For any $g,n \geq 0$, let $\mathcal{S}_{g,n}$ be the moduli space of spherical surfaces of genus $g$ with $n$ conical singularities (up to isometry), with a topology induced by Gromov-Hausdorff metric. We say that a spherical surface $S$ belongs to the discriminant $\Delta \subset$ if there is an interior angle equal to $\pi$ in an exterior domain $D$ of $X$.
\end{defn}

In spherical surfaces that do not belong to the discriminant, possible shapes of exterior domains are more constrained. In particular, half-spheres and foliated digons are impossible. Therefore, the bound of Theorem 4.12 is satisfied by spherical surfaces that do not belong to the discriminant.

\begin{thm}
Provided $2g+n \geq 3$, for any spherical surface $S$ in $\mathcal{S}_{g,n} \setminus \Delta$, any exterior domain is the complement of a polygon contained in an open half-sphere. Besides, there is an open neighborhood $V$ of $S$ in $\mathcal{S}_{g,n} \setminus \Delta$ such that for any spherical surface $T \in V$, the topological pair $(T,\mathcal{C}(T))$ is homeomorphic to $(S,\mathcal{C}(S))$.
\end{thm}

\begin{proof}
Hypothesis $2g+n \geq 3$ just eliminates the case of singular spheres. Thus, Theorem 4.6 implies that $D$ is either a foliated digon of angle at least $\pi$ or a convex polygon contained in a closed half-sphere.\newline
If $S$ does not belong to the discriminant, then for any exterior domain $D$, any interior angle of $D$ is distinct from $\pi$. Extremal arcs of a foliated digon contain at least one conical singularity (otherwise they can be extended and we eliminate the case of singular surfaces for which such a family is cyclic). The interior angle at such conical singularities is $\pi$. Therefore, $D$ cannot be a foliated digon. Condition on angles also implies that $D$ cannot be a half-sphere (where interior angles are equal to $\pi$). A spherical polygon (the complement of $D$ in a spherical chart) with angles strictly smaller than $\pi$ is convex and contained in an open half-sphere.\newline

The core $\mathcal{C}(S)$ is a compact. Therefore, there is a bound $m_{S}<\frac{\pi}{2}$ such that the injectivity radius of any point of $\mathcal{C}(S)$ is at most $m_{S}$. Consequently, there is a neighborhood of $S$ in the moduli space where no embedded half-sphere can appear in the middle of the core.\newline
There are finitely many exterior domains in $S$ and they are bounded by finitely many arcs. There is neighborhood in the moduli space where the lengths if the boundary arcs stay strictly smaller than $\pi$ and where the magnitudes of interior angles stay strictly bigger than $\pi$; The maximal distance of any conical singularity to any boundary arc of an exterior domain can also be controlled. Therefore, there is a neighborhood $V$ in the moduli space where exterior domains cannot appear, disappear, merge or be modified by creation of a new boundary arc. The topological pairs formed by the embedding of the core inside the surface are thus structurally stable.
\end{proof}

The topological pair $(\mathcal{C}(S),S)$ is invariant for any deformation inside a connected component of $\mathcal{S}_{g,n} \setminus \Delta$. Discriminant defines a walls-and-chambers structure on the moduli space (and any stratum obtained by fixing the conical angle of every singularity).
It would be interesting to refine the classification of spherical surfaces with the same conical singularities by fixing also the topological pair $(\mathcal{C}(S),S)$.\newline

\textit{Acknowledgements.} The author is grateful to Alexandre Eremenko, Dmitry Novikov, Dmitri Panov and the anonymous referee for their valuable remarks. The author is supported by the Israel Science Foundation (grant No. 1167/17) and the European Research Council (ERC) under the European Union Horizon 2020 research and innovation programme (grant agreement No. 802107).\newline

\nopagebreak
\vskip.5cm
\end{document}